\documentclass[11pt,a4paper]{article}

\usepackage[utf8]{inputenc}
\usepackage[T1]{fontenc}
\usepackage[english]{babel}
\usepackage{amssymb,amsmath,amsthm}
\usepackage{mathrsfs}
\usepackage{xspace}
\usepackage{url}
\usepackage{tikz}
\usepackage[margin=2.8cm]{geometry}

\theoremstyle{plain}
\newtheorem{theorem}{Theorem}[section]

\newtheorem{lemma}[theorem]{Lemma}

\theoremstyle{definition}
\newtheorem{remark}[theorem]{Remark}

\newcommand{\ds}{\displaystyle}

\newcommand{\opA}{ \mathcal{A} }

\newcommand{\opC}{ \mathcal{C} }
\newcommand{\opQ}{ \mathcal{Q} }
\newcommand{\opS}{ \mathcal{S} }
\newcommand{\opPl}{ {\mathcal{P}}_{\lambda_l} }
\newcommand{\proj}{ {\mathcal{P}}_{\mu_j} }

\newcommand{\pt}{ {\partial}_{t} }
\newcommand{\dn}{ {\partial}_{n} }
\newcommand{\dive}[1]{ \mathrm{div} \, #1 }
\newcommand{\dom}[1]{ \mathcal{D}\left(#1\right) }
\newcommand{\spec}[1]{ \sigma\left(#1\right) }
\newcommand{\res}[2]{\mathcal{R} \left(#1;#2\right)}

\newcommand{\Deltavec}{ \overrightarrow{\Delta} }

\newcommand{\conj}[1]{ \overline{#1} }
\newcommand{\re}[1]{ \mathfrak{Re} \, #1 }

\newcommand{\norm}[2]{{\left\|#1\right\|}_{#2}}
\newcommand{\ps}[3]{ {\left\langle #1 , #2 \right\rangle}_{#3} }
\newcommand{\abs}[1]{ \left|#1\right| }

\newcommand{\image}[1]{\mathrm{Im} \, #1}
\newcommand{\rank}{ \mathrm{rank} \,}
\newcommand{\vect}{ \mathrm{span} }

\newcommand{\field}[1]{\ensuremath{\mathbb{#1}}}
\newcommand{\C}{\field{C}\xspace}
\newcommand{\R}{\field{R}\xspace}
\newcommand{\ens}[1]{ \left\{#1\right\} }
\renewcommand{\sin}[1]{ \, \mathrm{sin}\left(#1\right)}

\title{Boundary approximate controllability of some linear parabolic systems}

\author{Guillaume Olive\thanks{LATP, UMR 7353, Aix-Marseille universit\'e,
Technop\^ole Ch\^ateau-Gombert, 39, rue F. Joliot-Curie,
13453 Marseille cedex 13, France.}}

\date{}

\begin{document}

\maketitle

\begin{abstract}
This paper focuses on the boundary approximate controllability of two classes
of linear parabolic systems, namely a system of $n$ heat equations coupled
through constant terms and a $2 \times 2$ cascade system coupled by means of a
first order partial differential operator with space-dependent coefficients.
For each system we prove a sufficient condition in any space dimension and we
show that this condition turns out to be also necessary in one dimension with
only one control. For the system of coupled heat equations we also study the
problem on rectangle, and we give characterizations depending on the position
of the control domain. Finally, we exhibit a cascade system for which the
distributed controllability holds whereas the boundary controllability does not.
The method relies on a general characterization due to H.O. Fattorini.
\end{abstract}



\medskip
\section{Introduction}

The controllability of parabolic systems is a difficult problem.
While Carleman estimates have been successfully used to prove the distributed null-controllability of some linear parabolic systems (e.g. \cite{AKBDGB}, \cite{GBdT}, \cite{Gue}, \cite{Mau}, \cite{BCGdT}), there are still many cases where these estimates appear to be of no help.
An example of such situation is when the control domain and the coupling domain do not meet each other (\cite{ABL}, \cite{RdT}).
The boundary controllability is another of these situations and requires new techniques to be solved.
In \cite{FCGBdT} and \cite{AKBGBdT}, the authors developped the method of moments of H.O. Fattorini and D.L. Russell to establish a characterization of the boundary null-controllability in dimension 1 for a system of $n$ coupled heat equations.
In \cite{ABL}, the authors used transmutation techniques to obtain a boundary null-controllability result in any dimension for a system of $2$ heat equations, with a particular coupling.
Finally, in \cite{KdT} the authors proved the boundary approximate controllability of a cascade system of $2$ heat equations in any dimension by developping the solution into Fourier series.
To the author knowledge, these results are the only ones concerning the boundary controllability of linear parabolic systems of heat-type.
For more details, a good account on actual methods and recent open problems for the distributed or boundary controllability of linear parabolic systems we refer to the survey \cite{AKBGBdTsurvey}.

In the present work we are interested in the boundary approximate controllability of two classes of linear parabolic systems introduced in \cite{FCGBdT} and \cite{KdT}.
More precisely, the first system we study is the following
\footnote{$\Deltavec$ denotes the vectorial Laplacian, in constrast with $\Delta$ for the scalar Laplacian.}
\begin{equation}\label{syst 1}
\left\{\begin{array}{rcll}
\ds \pt y &=&  \Deltavec y+Ay&\mbox{ in } (0,T) \times \Omega. \\
y&=& 1_{\gamma}Bg &\mbox{ on } (0,T) \times \partial \Omega. \\
y(0)&=& y_0 & \mbox{ in } \Omega.
\end{array}\right.
\end{equation}
where $T>0$, $\Omega$ is a bounded open subset of $\R^N$, assumed regular enough, $y$ is the state, $y_0$ is the initial data, $A$ and $B$ are $n \times n$ and $n \times m$ constant matrices with complex coefficients, $g$ is the control, to be searched in $L^2(0,T;L^2(\partial\Omega)^m)$ - so that in fact we have $m$ controls - and $\gamma \subset \partial\Omega$ is the control domain.

First of all, let us recall some basic facts about this kind of systems and their controllability properties:
\begin{enumerate}
\item
System (\ref{syst 1}) is well-posed in the following sense: for every $y_0 \in H^{-1}(\Omega)^n$ and $g \in L^2(0,T;L^2(\partial\Omega)^m)$, there exists a unique solution defined by transposition $y \in C^0([0,T];H^{-1}(\Omega)^n) \cap L^2(0,T;L^2(\Omega)^n)$ that depends continuously on the initial data $y_0$ and the control $g$.
\item
System (\ref{syst 1}) is said to be approximately controllable at time $T$ if for every $y_0,y_1 \in H^{-1}(\Omega)^n$ and every $\epsilon>0$, there exists a control $g \in L^2(0,T;L^2(\partial\Omega)^m)$ such that the corresponding solution $y$ satisfies
$$\norm{y(T)-y_1}{H^{-1}(\Omega)^n} \leq \epsilon.$$
We say that system (\ref{syst 1}) is approximately controllable if it is approximately controllable at time $T$ for every $T>0$.
\item
It is nowadays well-known that the controllability has a dual concept called observability and that they are linked by the following result: system (\ref{syst 1}) is approximately controllable at time $T$ if and only if its adjoint system
\footnote{Since the data are more regular and the system is autonomous, the solution can be taken in the sense of semigroups: $z(t)=\opS(t)z_0$, where $\opS(t)$ is the semigroup generated on $L^2(\Omega)^n$ by the operator $\Deltavec+A^*$ with domain $H^2(\Omega)^n \cap H^1_0(\Omega)^n$.
Let us recall that, for $z_0 \in H^1_0(\Omega)^n$, we have $z \in C^0([0,T];H^1_0(\Omega)^n) \cap L^2(0,T;H^2(\Omega)^n \cap H^1_0(\Omega)^n)$.}
$$
\left\{\begin{array}{rcll}
\ds \pt z &=&  \Deltavec z+A^*z&\mbox{ in } (0,T) \times \Omega. \\
z&=& 0 &\mbox{ on } (0,T) \times \partial \Omega. \\
z(0)&=& z_0 & \mbox{ in } \Omega.
\end{array}\right.
$$
is approximately observable at time $T$, that is it verifies the following unique continuation property
\begin{equation}\label{ucp nat}
\forall z_0 \in H^1_0(\Omega)^n, \quad
\bigg( B^*1_{\gamma}\dn z(t)=0 \mbox{ for a.e. } t \in (0,T) \bigg)
\Longrightarrow z_0=0.
\end{equation}
\end{enumerate}

The boundary controllability problem for (\ref{syst 1}) has been introduced in \cite{FCGBdT}.
In this paper, the authors proved a necessary and sufficient condition for this system to be null-controllable, and the same condition also characterizes the approximate controllability, see \cite[Theorem 1.1]{FCGBdT} and \cite[Theorem 5.2]{FCGBdT}.
We point out that this work has been done in 1D and for $2$ equations.
A generalization to the case of $n$ equations can be found in \cite{AKBGBdT}, still in 1D.
To the author knowledge, the only result that can be applied to system (\ref{syst 1}) in any dimension is \cite[Corollary 2.2]{ABL}, but the matrix $A$ has to have a very particular structure and it requires a geometric condition on $\gamma$.

In this paper we will provide conditions for the approximate controllability of this system in several interesting particular cases, see the sections \ref{sect cond suffi syst1} to \ref{sect appli rect} below.
Some results are already known but we give new and simpler proofs.

The second system we deal with is the following
\begin{equation}\label{syst 2}
\left\{\begin{array}{ll}
\ds \pt y_1 =  \Delta y_1 &\mbox{ in } (0,T) \times \Omega. \\
\ds \pt y_2 =  \Delta y_2 + G(x) \cdot \nabla y_1 + a(x) y_1&\mbox{ in } (0,T) \times \Omega. \\
y_1 = 1_{\gamma}g, \quad y_2=0 &\mbox{ on } (0,T) \times \partial\Omega. \\
y_1(0)= y_{1,0}, \quad y_2(0)=y_{2,0} & \mbox{ in } \Omega.
\end{array}\right.
\end{equation}
where $G \in W^{1,\infty}(\Omega)^N$, $a \in L^{\infty}(\Omega)$, and $g$ is still the control, but this time we only have one control: $g \in L^2(0,T;L^2(\partial\Omega))$.

The interest in the controllability of such systems started with \cite{KdT}.
In this paper the authors gave sufficient conditions for the approximate controllability.

In the present work we bring a new point of view to treat this problem.
This allows us to recover the result of \cite{KdT} and also to provide a necessary and sufficient condition in the 1D case.

The main tool to achieve our goals will be the use of a theorem of H.O. Fattorini.
In fact, in 1966, H.O. Fattorini gave an interesting characterization of the approximate controllability under a general abstract framework.
In his paper \cite{Fat} he proved that, under some reasonable assumptions, the only observation of the eigenfunctions completely characterizes the approximate controllability.
Actually, this theorem has been proved for bounded observation operators but it can easily be generalized to the case of relatively bounded observation operators as follows:

\begin{theorem}\label{fat}\label{FAT}
Let $H$ and $U$ be some complex Hilbert spaces.
Assume that $\opA : \dom{\opA} \subset H \longrightarrow H$ generates a strongly continuous semigroup $\opS(t)$ on $H$, has a compact resolvent, and the system of root vectors of its adjoint $\opA^*$ is complete in $H$.
Let $\opC: \dom{\opC} \subset H \longrightarrow U$ be relatively bounded with respect to $\opA$.
Then, we have the property
\begin{equation}\label{ucp fat}
\forall z_0 \in \dom{\opA}, \quad
\bigg( \opC \opS(t)z_0=0 \mbox{ for a.e. } t \in (0,+\infty) \bigg)
\Longrightarrow z_0=0,
\end{equation}
if and only if
$$\ker(s-\opA) \cap \ker \opC=\ens{0}, \quad \forall s \in \C.$$
\end{theorem}

We give a proof of this theorem (which slightly changes from the one of \cite{Fat}) in the appendix \ref{annexe 1}.

\begin{remark}\label{rem prem}
Note that the condition $\ker(s-\opA) \cap \ker \opC=\ens{0}$ can also be formulated as
$$
\forall z_0 \in \ker(s-\opA), \quad
\bigg( \opC \opS(t)z_0=0 \mbox{ for a.e. } t \in (0,+\infty) \bigg)
\Longrightarrow z_0=0.
$$
We use the first formulation because it is more eloquent, see Remark \ref{rem dim finie} below, but the most important is that Theorem \ref{fat} states that, in order to verify the property (\ref{ucp fat}) it is enough to do so only on the eigenspaces of $\opA$.
\end{remark}

\begin{remark}\label{semigpe ana}
We will see that the operators $\opA$ we consider generate an analytic semigroup.
For instance, for the first system we shall apply Theorem \ref{fat} to $\opA=\Deltavec+A^*$.
It follows from this property that $z(\cdot)=\opS(\cdot)z_0$ is analytic in time and has a regularizing effect ($\opS(t)z_0 \in \dom{\opA^{\infty}}$ as soon as $t>0$, even for $z_0 \in H$).
This allows us to replace in (\ref{ucp fat}) the interval $(0,+\infty)$ by any interval $(0,T)$, $T>0$, and to take the data $z_0$ in any space that at least contains $\dom{\opA^{\infty}}$.
This shows that (\ref{ucp nat}) and (\ref{ucp fat}) are equivalent properties.
In partictular, we see that the approximate controllability of our systems is independent of the time of control $T$.
\end{remark}

\begin{remark}\label{rem dim finie}
When $H=\mathbb{C}^n$ and $U=\mathbb{C}^m$, $\opA=A^*$ and $\opC=B^*$ (where $A$ and $B$ are still contant matrices) this theorem can be used to prove that the ordinary differential system
$$
\left\{\begin{array}{rcll}
\ds \frac{d}{dt} y &=& Ay +Bg &\mbox{ in } (0,T). \\
\ds y(0)&=& y_0 &
\end{array}\right.
$$
is controllable\footnote{In finite dimension all the notions of controllability are equivalent.} if and only if
$$\ker(s-A^*) \cap \ker B^*=\ens{0}, \quad \forall s \in \C.$$
This characterization is nowadays known as the Hautus test (despite it has been proved earlier by H.O. Fattorini).
M.L.J Hautus gave a direct proof of the equivalence with another characterization, the well-known Kalman rank condition (see \cite[Theorem 1', \S 2]{Hau})
$$\rank{(B | AB | A^2B | \cdots | A^{n-1}B)}=n.$$
\end{remark}

Finally, let us mention the recent work \cite{BT} where the authors also extended the theorem of \cite{Fat} in view of the stabilizability of some other parabolic systems.

\paragraph{Notations}
We denote by $\ens{-\lambda_l}_l$ the distinct Dirichlet eigenvalues of $\Delta$ on $\Omega$.
For each $l$, we denote by $\ens{\phi_{l,m}}_m$ an orthonormal basis in $L^2(\Omega)$ of the eigenspace of $\Delta$ associated with the eigenvalue $-\lambda_l$, and by $m_l$ the dimension of this eigenspace.
It can be verified that all the following results are independent of the choice of the basis $\ens{\phi_{l,m}}_m$.

In section \ref{result syst2}, we use the notation $\opPl$ for the orthogonal projection in $L^2(\Omega)$ on the eigenspace of $\Delta$ associated with $-\lambda_l$, that is $\opPl u= \sum_{m=1}^{m_l} \ps{u}{\phi_{l,m}}{L^2(\Omega)} \phi_{l,m}$, for $u \in L^2(\Omega)$.

In sections \ref{sect 1D}, \ref{sect B vect} and \ref{sect 1D syst2}, we consider the 1D case. In particular $m_l=1$ so that, for commodity, we simply use the notation $\phi_l$ instead of $\phi_{l,1}$.

In section \ref{sect appli rect}, we use the notation $-\lambda_l^{X_1}$ (resp. $-\lambda_l^{X_2}$) to emphasize that this is the eigenvalues corresponding to the domain $\Omega=(0,X_1)$ (resp. $\Omega=(0,X_2)$), and we denote by $\phi_l^{X_1}$ (resp. $\phi_l^{X_2}$) a corresponding eigenfunction.

\section{Results for the first system}\label{sect A}

We start by applying Theorem \ref{fat} to the operators
$$\opA=\Deltavec+A^*, \quad \dom{\opA}=H^2(\Omega)^n \cap H^1_0(\Omega)^n,$$
and
$$\opC=B^* 1_{\gamma}\dn, \quad \dom{\opC}=H^2(\Omega)^n \cap H^1_0(\Omega)^n.$$

By a perturbation argument we can check that $\opA$ generates an analytic semigroup on $L^2(\Omega)^n$, has a compact resolvent and the system of root vectors of $\opA^*$ is complete in $L^2(\Omega)^n$ (using, for instance, the Keldysh's perturbation theorem, see \cite[Theorem 4.3, Chapter I, \S 4]{Mar}), so that it satisfies the required hypothesis.
On the other hand, the operator $\opC$ is clearly relatively bounded with respect to $\opA$.

Thus, system (\ref{syst 1}) is approximately controllable (at some time or at any time, see Remark \ref{semigpe ana}) if and only if
\begin{equation}\label{HT}
\ker(s-(\Deltavec+A^*)) \cap \ker (B^*1_{\gamma}\dn)=\ens{0}, \quad \forall s \in \C.
\end{equation}

To describe the spectral elements of $\Deltavec+A^*$ we introduce the following notations:

\paragraph{Notations}
We denote by $\ens{\theta_i}_i \subset \C$ the distinct eigenvalues of the matrix $A^*$ and, for each $i$, by $\ens{w_{i,j}}_j \subset \C^{n}$ a basis of $\ker(\theta_i-A^*)$.

In view of section \ref{sect 1D}, we also denote by $m_i$ the dimension of $\ker(\theta_i-A^*)$ and we define
$$P_i=\left(\begin{array}{c} w_{i,1} | \cdots | w_{i,m_i} \end{array}\right).$$

One can check that all the following results do not depend on the choice of the basis $\ens{w_{i,j}}_j$.

These notations in mind, it is not difficult to see that the spectrum of $\Deltavec+A^*$ is
$$\spec{\Deltavec+A^*}=\ens{-\lambda_l+\theta_i}_{l,i},$$
and its eigenspaces are
$$\ker(s-(\Deltavec+A^*))=\vect \ens{w_{i,j} \phi_{l,m}}_{\substack{i,j,l,m \\ -\lambda_l+\theta_i=s }}.$$

As we can see, the spectral structure of the operator $\Deltavec +A^*$ is somehow separated into a scalar differential part and a vectorial algebraic part.
Moreover, the operator $\opC$ we consider is $\opC=B^* 1_{\gamma}\dn$, and $1_{\gamma}\dn$ acts on the scalar differential part while $B^*$ acts on the vectorial algebraic part (recall that $B$ is a constant matrix).
In this particular situation we have good hopes to obtain an easier characterization than condition (\ref{HT}).
This is what establish the results in the following sections.

\begin{remark}\label{rem vp db}
We shall emphasize that the eigenvalues $-\lambda_l+\theta_i$  are not necessarily distinct.
All along this work, for an eigenvalue $s \in \spec{\Deltavec+A^*}$, we will denote by $l^s_1, \ldots, l^s_{r_s}$ and $i^s_1, \ldots, i^s_{r_s}$ (with possibly $r_s=1$) all the distinct indices such that
$$s=-\lambda_{l_{1}^{s}} +\theta_{i_{1}^{s}}=\ldots=-\lambda_{l_{r_s}^{s}} +\theta_{i_{r_s}^{s}}.$$
Note that $r_s<+\infty$ since there is a finite number of $\theta_i$.

As as result, any $u \in \ker(s-(\Deltavec+A^*))$ has a writing of the form
$$u=\sum_{k=1}^{r_s} \sum_{j,m} \alpha_{k,j,m} w_{i_k^s,j} \phi_{l_k^s,m},$$
for some $\alpha_{k,j,m} \in \C$.

Since we will always reason at $s$ fixed, we will omit the dependence with respect to $s$ during the proofs (for the sake of clarity), though we will keep this notation in the statements of the results.
\end{remark}

\subsection{A sufficient condition}\label{sect cond suffi syst1}

As noticed in Remark \ref{rem vp db} it may happen that some eigenvalue $s$ can be written as $s=-\lambda_l+\theta_i=-\lambda_{l'}+\theta_{i'}$ with $i' \neq i$ and $l' \neq l$.
This phenomenon of "resonance" is a consequence of the coupling (the matrix $A$) and as a result is specific to the fact that we study a system, in contrast with a single equation.
We will see that all the difficulties will precisely come from this point.
This fact has been highlighted for the very first time in \cite{FCGBdT}.
The following theorem shows that, when there is no phenomenon of resonance, the controllability is simply reduced to an algebraic condition, whatever the space dimension $N$ and the control domain $\gamma$ are.

\begin{theorem}\label{prop cond suffi}
Assume that for every eigenvalue $s \in \spec{\Deltavec+A^*}$ we have $r_s=1$.
Then, the ND system
$$
\left\{\begin{array}{rcll}
\ds \pt y &=&  \Deltavec y+Ay&\mbox{ in } (0,T) \times \Omega. \\
y&=& 1_{\gamma}Bg &\mbox{ on } (0,T) \times \partial \Omega. \\
y(0)&=& y_0 & \mbox{ in } \Omega.
\end{array}\right.
$$
is approximately controllable if and only if
\begin{equation}\label{HT alg}
\ker(\theta_i-A^*) \cap  \ker B^*=\ens{0}, \quad \forall i.
\end{equation}
\end{theorem}

In general, the assumption of this theorem is not a necessary condition, except in some very particular cases, see Theorem \ref{1D B vect} in section \ref{sect B vect}.

\begin{remark}
Condition (\ref{HT alg}) is nothing but the condition of Theorem \ref{fat} on the algebraic part of the system (see also Remark \ref{rem dim finie}).
We would also expect to require the similar condition concerning the scalar differential part, namely
\begin{equation}\label{HT diff}
\ker(-\lambda_l-\Delta) \cap\ker( 1_{\gamma}\dn)=\ens{0}, \quad \forall l,
\end{equation}
but actually this condition is always fulfilled, see \cite[Lemma]{MMS}, so that it is implicitly hidden in the theorem (and this will be used in the proof).
This condition corresponds to the approximate controllability of the heat equation from the boundary.
\end{remark}

\begin{remark}\label{rem 1vp}
An easy but nontheless interesting consequence of Theorem \ref{prop cond suffi} is when $A^*$ has only one eigenvalue.
In this case the assumption is naturally satisfied.
This permits for instance to easily prove that a ND cascade system
$$
\left\{\begin{array}{rcll}
\ds \pt y &=&  \Deltavec y+
\left(\begin{array}{cccc}
0 & \cdots & \cdots & 0 \\
a_{21} & \ddots & & \vdots \\
\vdots & \ddots &  \ddots & \vdots \\
a_{n,1} & \cdots & a_{n,n-1} & 0
\end{array} \right)
y&\mbox{ in } (0,T) \times \Omega. \\
y&=& 1_{\gamma} \left(\begin{array}{c} 1 \\ 0 \\ \vdots \\ 0 \end{array}\right) g &\mbox{ on } (0,T) \times \partial \Omega. \\
y(0)&=& y_0 & \mbox{ in } \Omega.
\end{array}\right.
$$
is approximately controllable if and only if $a_{i,i-1} \neq 0$ for every $2 \leq i \leq n$.
\end{remark}

Theorem \ref{prop cond suffi} is a straightforward consequence of the following two lemma.
The first lemma shows that condition (\ref{HT alg}) is always a necessary condition for the approximate controllability of system (\ref{syst 1}), while the second lemma shows that this condition is also enough to "control" the eigenvalues $s$ for which $r_s=1$.
Since we assume that there are only such eigenvalues, Theorem \ref{prop cond suffi} will be proved.

\begin{lemma}\label{lemma cond nec}
If system (\ref{syst 1}) is approximately controllable, then (\ref{HT alg}) holds.
\end{lemma}

\begin{lemma}\label{lemma cond suffi}
Assume that (\ref{HT alg}) holds.
Then, for any eigenvalue $s \in \spec{\Deltavec+A^*}$ such that $r_s=1$, we have
$$
\ker(s-(\Deltavec+A^*)) \cap \ker(B^* 1_{\gamma}\dn)=\ens{0}.
$$
\end{lemma}

\begin{proof}[Proof of Lemma \ref{lemma cond nec}]
Let $w \in \ker(\theta_i-A^*) \cap \ker B^*$.
Let $\lambda \in \spec{\Delta}$.
Taking any nonzero $\phi \in \ker(\lambda-\Delta)$ we see that $u=\phi w$ belongs to $\ker(\lambda+\theta_i-(\Deltavec+A^*)) \cap \ker (B^*1_{\gamma}\dn)$, so that $u=0$ by assumption, and thus also $w=0$.
\end{proof}

\begin{proof}[Proof of Lemma \ref{lemma cond suffi}]
Let $s \in \spec{\Deltavec+A^*}$ with $r_s=1$, $u \in \ker(s-(\Deltavec+A^*))\cap \ker(B^* 1_{\gamma}\dn)$.
Since $r_s=1$, $u$ writes
$$u=\sum_{j,m} \alpha_{j,m} w_{i_1,j}\phi_{l_1,m}$$
for some $\alpha_{j,m} \in \C$.
Let us set $\beta_j=1_{\gamma}\dn \left(\sum_{m}\alpha_{j,m} \phi_{l_1,m}\right) \in L^2(\partial\Omega)$ so that we have
$$B^*\left(\sum_{j} \beta_j w_{i_1,j}\right)=0.$$

Since $\sum_{j} \beta_j w_{i_1,j} \in \ker(\theta_{i_1} -A^*)$ we can use (\ref{HT alg}) to obtain $\sum_{j} \beta_j w_{i_1,j}=0$.
Using the linear independance of $\ens{w_{i_1,j}}_{j}$ we deduce that $\beta_j=0$ for every $j$, that is
$$1_{\gamma}\dn \left(\sum_{m}\alpha_{j,m}\phi_{l_1,m}\right)=0, \quad \forall j.$$

Since $\sum_{m}\alpha_{j,m}\phi_{l_1,m} \in \ker(-\lambda_{l_1}-\Delta)$, using now (\ref{HT diff}) gives $\sum_{m}\alpha_{j,m}\phi_{l_1,m}=0$.
Thanks to the linear independance of $\ens{\phi_{l_1,m}}_m$ we conclude that $\alpha_{j,m}=0$ for every $j,m$, that is $u=0$.
\end{proof}

\subsection{As many controls as equations}\label{asmcase}

As a second result we recover the known fact (see \cite[Theorem 5.3]{FCGBdT}) that we can control the system from the boundary if we put as many controls as equations.
In this particular case, the coupling becomes inconsequential (the matrix $A$ can even be $A=0$, that is no coupling at all).
This situation can be understood as $n$ uncoupled equations with one control for each.
This result has been obtained in \cite{FCGBdT} by means of a Carleman estimate but we provide here an alternative proof, which is also simpler in our case.

\begin{theorem}
The ND system
$$
\left\{\begin{array}{rcll}
\ds \pt y &=&  \Deltavec y+Ay&\mbox{ in } (0,T) \times \Omega. \\
y&=& 1_{\gamma}Bg &\mbox{ on } (0,T) \times \partial \Omega. \\
y(0)&=& y_0 & \mbox{ in } \Omega.
\end{array}\right.
$$
is approximately controllable if we assume that
$$\ker B^*=\ens{0}.$$
\end{theorem}

\begin{proof}[Proof]
Let $s$ be an eigenvalue of $\Deltavec +A^*$ and $u \in \ker(s-(\Deltavec+A^*))\cap \ker (B^*1_{\gamma}\dn)$.
Then, $u$ writes
$$u=\sum_{k=1}^{r} \sum_{j,m} \alpha_{k,j,m} w_{i_k,j} \phi_{l_k,m}$$
for some $\alpha_{k,j,m} \in \C$.
Since $u \in \ker (B^*1_{\gamma}\dn)$ and $\ker B^*=\ens{0}$ by assumption, we have
$$\sum_{k,j} \left(\sum_{m} \alpha_{k,j,m} 1_{\gamma}\dn\phi_{l_k,m} \right) w_{i_k,j}=0.$$

By the linear independence of $\ens{w_{i,j}}_{i,j}$ we obtain
$$1_{\gamma}\dn \left(\sum_{m} \alpha_{k,j,m} \phi_{l_k,m}\right)=0, \quad \forall k, \forall j.$$

Since $\sum_{m} \alpha_{k,j,m} \phi_{l_k,m} \in \ker(-\lambda_{l_k}-\Delta)$ we deduce that $\sum_{m} \alpha_{k,j,m} \phi_{l_k,m}=0$ (using (\ref{HT diff})), and by the linear independence of $\ens{\phi_{l,m}}_m$ it follows that $\alpha_{k,j,m}=0$ for every $k,j,m$, that is $u=0$.
\end{proof}

\subsection{The 1D case}\label{sect 1D}

The 1D case is a very particular situation because the boundary is reduced to two points, $\ens{0}$ and $\ens{L}$, if $\Omega=(0,L)$.
In particular, only three possibilities arise for $\gamma$, namely $\gamma=\ens{0}$, $\gamma=\ens{L}$ or $\gamma= \ens{0} \cup \ens{L}$.
We will study these three cases.

The results of this section have already been obtained in \cite{AKBGBdT}, with another formulation though, and a different proof.

We start with the case $\gamma=\ens{0}$ (we refer to the beginning of section \ref{sect A} for the notations):

\begin{theorem}\label{ex 1D 1}
The 1D system
$$
\left\{\begin{array}{rcll}
\ds \pt y &=&  \Deltavec y+Ay&\mbox{ in } (0,T) \times (0,L). \\
y&=& 1_{\ens{0}}Bg &\mbox{ on } (0,T) \times \ens{0,L}. \\
y(0)&=& y_0 & \mbox{ in } (0,L).
\end{array}\right.
$$
is approximately controllable if and only if, for every $s \in \spec{\Deltavec+A^*}$, we have
$$\rank (B^* P_{i^s_1} | \cdots | B^* P_{i^s_{r_s}})=\sum_{k=1}^{r_s} m_{i_k}.$$
\end{theorem}

\begin{proof}[Proof]
Let $u \in \ker(s-(\Deltavec+A^*)) \cap  \ker B^* 1_{\ens{0}}\dn$, where $s \in \spec{\Deltavec+A^*}$.
We know that $u$ writes
$$u=\sum_{k=1}^{r} \sum_{j,m} \alpha_{k,j,m} w_{i_k,j} \phi_{l_k}$$
for some $\alpha_{k,j,m} \in \C$ and we have
$$\sum_{k=1}^{r} \sum_{j=1}^{m_{i_k}} \alpha_{k,j} B^*w_{i_k,j}  \phi_{l_k}'(0)=0.$$

This implies that $\alpha_{k,j}=0$ for every $k,j$ if and only if the matrix
$$\left(\begin{array}{ccccc} \phi_{l_1}'(0)B^*P_{i_1} & \big| &\cdots & \big| & \phi_{l_r}'(0)B^*P_{i_r} \end{array}\right)$$
has full rank, that is
$$\rank \left(\begin{array}{ccccc} \phi_{l_1}'(0)B^*P_{i_1} & \big| &\cdots & \big| & \phi_{l_r}'(0)B^*P_{i_r} \end{array} \right)=\sum_{k=1}^{r} m_{i_k}.$$

To conclude it remains to observe that
\begin{multline*}
\rank\left( \begin{array}{ccccc} \phi_{l_1}'(0)B^*P_{i_1} & \big| &\cdots & \big| & \phi_{l_r}'(0)B^*P_{i_r} \end{array}\right) \\
=\rank \left(\begin{array}{ccccc} B^*P_{i_1} & \big| &\cdots & \big| & B^*P_{i_r} \end{array} \right)
\end{multline*}
since $\phi_l'(0) \neq 0$ for every $l$.
\end{proof}

The same result holds if we consider $\gamma=\ens{L}$ instead of $\gamma=\ens{0}$.
When $\gamma$ is the whole boundary, that is $\gamma=\ens{0} \cup \ens{L}$, we have the following characterization:

\begin{theorem}\label{ex 1D 2}
The 1D system
$$
\left\{\begin{array}{rcll}
\ds \pt y &=&  \Deltavec y+Ay&\mbox{ in } (0,T) \times (0,L). \\
y&=& Bg &\mbox{ on } (0,T) \times \ens{0,L}. \\
y(0)&=& y_0 & \mbox{ in } (0,L).
\end{array}\right.
$$
is approximately controllable if and only if, for every $s \in \spec{\Deltavec+A^*}$, we have
$$\rank \left(\begin{array}{ccccc} \begin{array}{r} \phi_{l^s_1}'(0) B^* P_{i^s_1} \\ \phi_{l^s_1}'(L)  B^* P_{i^s_1} \end{array} & \Bigg| & \cdots & \Bigg| & \begin{array}{r} \phi_{l^s_{r_s}}'(0) B^* P_{i^s_{r_s}} \\ \phi_{l^s_{r_s}}'(L) B^* P_{i^s_{r_s}} \end{array} \end{array}\right)=\sum_{k=1}^{r_s} m_{i_k}.$$
\end{theorem}

The proof is the same as the proof of Theorem \ref{ex 1D 1}.

\begin{remark}
Further to these two theorems, we see that it may happen that system (\ref{syst 1}) is controllable with a control acting on both parts of the boundary whereas it is not controllable if the control only acts on one part.
Indeed, let us consider on $\Omega=(0,\pi)$ the system described by
$$A=
\left(\begin{array}{cc}
0 & -4 \\
1 & 5
\end{array} \right),
\quad
B= \left(\begin{array}{c} 1 \\ 0 \end{array}\right).
$$
We recall that the eigenvalues of $\Delta$ on $(0,\pi)$ are $-\lambda_l=-l^2$ and the corresponding eigenfunctions are $\phi_l(x)= \sqrt{\frac{2}{\pi}} \sin{l x}$.
We can check that
$$\spec{A^*}=\ens{\theta_1=1,\theta_2=4},$$
$$\ker(\theta_1-A^*)=\vect \ens{\left(\begin{array}{c} 1 \\ 1 \end{array}\right)},
\quad \ker(\theta_2-A^*)=\vect \ens{\left(\begin{array}{c} 1 \\ 4 \end{array}\right)}.$$
Since $r_s=1$ for every eigenvalue $s \in \spec{\Deltavec+A^*}$ except $s=-\lambda_1+\theta_1=-\lambda_2+\theta_2$, it is not difficult to check that the condition of Theorem \ref{ex 1D 2} is fulfilled, whereas the one of Theorem \ref{ex 1D 1} is not.
\end{remark}

\subsection{Only one control:{\mathversion{bold} $m=1$} }\label{sect B vect}

Another interesting situation is when we try to control system (\ref{syst 1}) with only one control.
This corresponds to $m=1$, so that the matrix $B$ is in fact a (column) vector.

\begin{remark}\label{mi egal un}
It is not difficult to see that the condition $\ker(\theta_i-A^*) \cap \ker B^*=\ens{0}$ is equivalent to (see the beginning of section \ref{sect A} for the notations)
$$\rank B^* P_i =m_i.$$
Thus, when $B^*$ has now only one line, we necessarily have
$$m_i=1.$$
In such a case, note also that $P_i$ is reduced to $w_{i,1}$, so that $B^*P_i$ is a scalar, and $\rank B^* P_i =1$ then simply means that this scalar is not zero.
\end{remark}

Let us come back to the 1D case.
We can always assume that $\ker(\theta_i-A^*) \cap \ker B^*=\ens{0}$ for every $i$ since it is a necessary condition (see Lemma \ref{lemma cond nec}).
According to Remark \ref{mi egal un}, we then know that $m_i=1$ and $B^*P_i$ is a nonzero scalar  for every $i$.
Thus, for every $s \in \spec{\Deltavec+A^*}$, we have
$$
\left\{\begin{array}{l}
\ds \rank \left( B^* P_{i^s_1} | \cdots | B^* P_{i^s_{r_s}} \right)=\rank (1 | \cdots | 1), \\
\ds \sum_{k=1}^{r_s} m_{i_k}=r_s.
\end{array}
\right.
$$

As a result, in this particular case which is $m=1$, Theorem \ref{ex 1D 1} becomes

\begin{theorem}\label{1D B vect}
Assume that $m=1$.
The 1D system
$$
\left\{\begin{array}{rcll}
\ds \pt y &=&  \Deltavec y+A y&\mbox{ in } (0,T) \times (0,L). \\
y&=& 1_{\ens{0}}Bg &\mbox{ on } (0,T) \times \ens{0,L}. \\
y(0)&=& y_0 & \mbox{ in } (0,L).
\end{array}\right.
$$
is approximately controllable if and only if the following two conditions hold:
\begin{enumerate}
\item $\ker(\theta_i-A^*) \cap  \ker B^*=\ens{0}$ for every $i$.
\item For every eigenvalue $s \in \spec{\Deltavec+A^*}$ we have $r_s=1$.
\end{enumerate}
\end{theorem}

This result is historically the first relevant difference between distributed and boundary controllability for parabolic systems (these properties are equivalent for the heat equation for instance).
This has been proved in \cite{FCGBdT}.
Moreover, this also shows that if this system is controllable with a boundary control then it is also controllable with a distributed control (recall that the distributed controllability of this system is characterized by only the first condition, see \cite{AKBDGB}).
We insist on the fact that this is a result in 1D; except in the framework of \cite{ABL}, the problem is open in higher space dimension.

We have a similar result for Theorem \ref{ex 1D 2} when $m=1$:

\begin{theorem}
Assume that $m=1$.
The 1D system
$$
\left\{\begin{array}{rcll}
\ds \pt y &=&  \Deltavec y+A y&\mbox{ in } (0,T) \times (0,L). \\
y&=& Bg &\mbox{ on } (0,T) \times \ens{0,L}. \\
y(0)&=& y_0 & \mbox{ in } (0,L).
\end{array}\right.
$$
is approximately controllable if and only if the following two conditions hold:
\begin{enumerate}
\item $\ker(\theta_i-A^*) \cap \ker B^* =\ens{0}$ for every $i$.
\item For every eigenvalue $s \in \spec{\Deltavec+A^*}$, either $r_s=1$, either $r_s=2$ and in this latter case:
$$\rank \left(\begin{array}{cc} \phi_{l^s_1} '(0) & \phi_{l^s_2}'(0) \\ \phi_{l^s_1} '(L) & \phi_{l^s_2}'(L) \end{array}  \right)=2.$$
\end{enumerate}
\end{theorem}

Finally, let us give a result in any dimension when $m=1$.

\begin{theorem}\label{prop B vect}
Assume that $m=1$.
The ND system
$$
\left\{\begin{array}{rcll}
\ds \pt y &=&  \Deltavec y+Ay&\mbox{ in } (0,T) \times \Omega. \\
y&=& 1_{\gamma}Bg &\mbox{ on } (0,T) \times \partial \Omega. \\
y(0)&=& y_0 & \mbox{ in } \Omega.
\end{array}\right.
$$
is approximately controllable if and only if the following two conditions hold:
\begin{enumerate}
\item $\ker(\theta_i-A^*) \cap \ker B^*=\ens{0}$ for every $i$.
\item
For every $s \in \spec{\Deltavec+A^*}$, we have
$$\bigg(\ker( -\lambda_{l^s_1} -\Delta) + \ldots + \ker(-\lambda_{l^s_{r_s}} -\Delta)\bigg) \cap  \ker( 1_{\gamma}\dn)=\ens{0}.$$
\end{enumerate}
\end{theorem}

Note that the second condition is relevant only for $s$ with $r_s >1$, see (\ref{HT diff}).

\begin{proof}[Proof of Theorem \ref{prop B vect}]
Let $s \in \spec{\Deltavec +A^*}$ and let $u \in \ker(s-(\Deltavec+A^*))\cap \ker(B^* 1_{\gamma}\dn)$.
We know that $u$ writes
$$u=\sum_{k=1}^{r} \sum_{j,m} \alpha_{k,j,m} w_{i_k,j} \phi_{l_k,m}$$
for some, $\alpha_{k,j,m} \in \C$ and we have
$$1_{\gamma}\dn \left(\sum_{k} \sum_{m} \beta_{k,m} \phi_{l_k,m}\right)=0,$$
where $\beta_{k,m}=\sum_{j} \alpha_{k,j,m} B^*w_{i_k,j}$.

Since $B^*$ is a row vector, $\beta_{k,m}$ is a scalar, so that we can use the second condition and obtain that $\sum_{m} \beta_{k,m} \phi_{l_k,m}=0$ for every $k$.
By the linear indepedence of $\ens{\phi_{l,m}}_m$ we obtain that $\beta_{k,m}=0$ for every $k,m$, that is
$$B^*\left(\sum_{j} \alpha_{k,j,m} w_{i_k,j}\right)=0, \quad \forall k,m.$$

Using now the first condition this gives $\sum_{j} \alpha_{k,j,m} w_{i_k,j}=0$ and it follows that $\alpha_{k,j,m}=0$ for every $k,j,m$, that is $u=0$.

Let us now show that these conditions are also necessary.
We only prove it for the second condition since it is already known for the first one, see Lemma \ref{lemma cond nec}.

Let $\phi=\phi_{l_1}+\ldots+\phi_{l_r}$, with $\phi_l \in \ker(-\lambda_l-\Delta)$, be such that $1_{\gamma}\dn \phi=0$.
For every $k$, let $w_{i_k}$ be any eigenvector of $A^*$ associated with $\theta_{i_k}$.
We know that $B^*w_i$ is a scalar and $B^*w_i \neq 0$ thanks to the first condition (we have just recalled that it is a necessary condition).
Thus, we can define
$$u=\frac{1}{B^*w_{i_1}}w_{i_1} \phi_{l_1}+\ldots+\frac{1}{B^*w_{i_r}} w_{i_r} \phi_{l_r}.$$
We can see that $u \in \ker(s-(\Deltavec+A^*)) \cap \ker(B^* 1_{\gamma}\dn) $ so that $u=0$ by assumption.
It follows from the linear independence of $\ens{w_i}_i$ that $\phi_{l_k}=0$ for every $k$, that is $\phi=0$.
\end{proof}

\subsection{On a rectangular domain}\label{sect appli rect}

In this section we still consider system (\ref{syst 1}) but the domain $\Omega$ is now a rectangle
$$\Omega=(0,X_1) \times (0,X_2).$$

We denote the faces of our rectangle by $\gamma_L$, $\gamma_R$, $\gamma_T$ and $\gamma_B$, as on Figure \ref{rectangle}:
\begin{figure}[!h]
\begin{center}
\begin{tikzpicture}
\draw (0,0) rectangle (6,3)  ;
\node at (3,1.5) {\LARGE $\Omega$};
\node at (-0.5,1.5) {\Large $\gamma_L$};
\node at (6.5,1.5) {\Large $\gamma_R$};
\node at (3,3.5) {\Large $\gamma_T$};
\node at (3,-0.5) {\Large $\gamma_B$};
\node at (0,-0.30) {$(0,0)$};
\node at (0,3.30) {$(0,X_2)$};
\node at (6,-0.30) {$(X_1,0)$};
\end{tikzpicture}
\caption{Domain $\Omega$ for section \ref{sect appli rect}.}
\label{rectangle}
\end{center}
\end{figure}
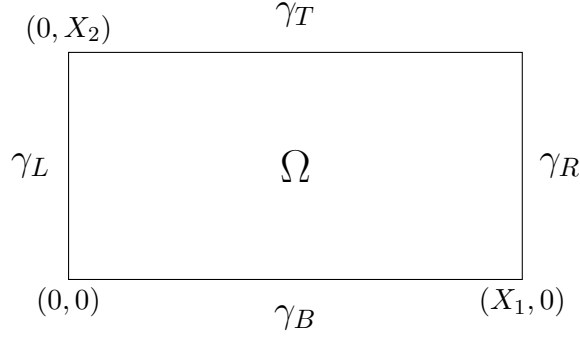

The goal of this section is to prove several results about the boundary controllability of system (\ref{syst 1}), by discussing on the geometric position of $\gamma$.

\begin{theorem}\label{prop rect 1}
If $\gamma=\gamma_L$, then system (\ref{syst 1}) is approximately controllable if and only if so is the following 1D system
\begin{equation}\label{syst rect 1D}
\left\{\begin{array}{rcll}
\ds \pt y &=&  \Deltavec_{x_1} y+Ay&\mbox{ in } (0,T) \times (0,X_1). \\
y&=& 1_{\ens{0}}Bg &\mbox{ on } (0,T) \times \ens{0,X_1}. \\
y(0)&=& y_0 & \mbox{ in } (0,X_1).
\end{array}\right.
\end{equation}

If $\gamma=\gamma_L\cup \gamma_R$ then system (\ref{syst 1}) is approximately controllable if and only if so is the following 1D system
$$
\left\{\begin{array}{rcll}
\ds \pt y &=&  \Deltavec_{x_1} y+Ay&\mbox{ in } (0,T) \times (0,X_1). \\
y&=& Bg &\mbox{ on } (0,T) \times \ens{0,X_1}. \\
y(0)&=& y_0 & \mbox{ in } (0,X_1).
\end{array}\right.
$$
\end{theorem}

We recall that the controllability of these 1D systems has been studied in sections \ref{sect 1D} and \ref{sect B vect}.

For the heat equation, a similar result has been established in \cite{Mil} for the null-controllability when $\gamma$ is one of the faces of $\partial\Omega$.

We consider next the case of two consecutive faces, with for instance $\gamma=\gamma_L \cup \gamma_T$.

\begin{theorem}\label{prop rect 2}
If $\gamma=\gamma_L \cup \gamma_T$ and $\ker(\theta_i-A^*) \cap \ker B^*=\ens{0}$ for every $i$, then system (\ref{syst 1}) is approximately controllable for $n=2$.
\end{theorem}

The geometry of $\gamma$ (including two different directions $\gamma_L$ and $\gamma_T$) is such that in some sense it "creates" an additional control.
Thus, everything happens as if we had two controls for two equations and we can expect the controllability to hold, as it is showed in section \ref{asmcase}.
Theorem \ref{prop rect 1} shows that this is not true if we pick two parallel faces $\gamma=\gamma_L \cup\gamma_R$.
When more equations are considered, the following counter-example strengthen this point of view.

\begin{theorem}\label{prop rect 3}
Even if $\gamma=\gamma_L \cup \gamma_T$ and $\ker(\theta_i-A^*) \cap \ker B^*=\ens{0}$ for every $i$, system (\ref{syst 1}) may be not approximately controllable when $n \geq 4$.
\end{theorem}

\begin{remark}
It is worth mentioning that, in all the previous statements, we can replace $\gamma_L$ (resp. $\gamma_R$, $\gamma_T$, $\gamma_B$) by a nonempty open part of it.
This is easily seen in the following proofs by using the analyticity of the 1D eigenfunctions of $\Delta$.
\end{remark}

The main ingredient that will make the proofs work is the following.
The (not necessarily distinct) eigenvalues of $\Delta=\Delta_{x_1}+\Delta_{x_2}$ on $(0,X_1) \times (0,X_2)$ are
$$
-\Lambda_{p,q}=-\lambda_p^{X_1}-\lambda_q^{X_2},
$$
and the corresponding eigenfunctions are
$$\Phi_{p,q}(x,y)=\phi_p^{X_1}(x)\phi_q^{X_2}(y), \quad x \in (0,X_1), y \in (0,X_2).$$

In this case the dimension $m_l$ of the eigenspace of $\Delta$ associated with $-\lambda_l$ is exactly the number of distinct couples of indices $(p,q)$ such that $\Lambda_{p,q}=\lambda_l$.
We denote by $(p_l^1,q_l^1), \ldots,$ $(p_l^{m_l},q_l^{m_l})$ all such indices.
Note that for every $m$ we necessarily have
\begin{equation}\label{q diff}
p_l^m \neq p_l^{m'} \quad \mbox{and} \quad q_l^m \neq q_l^{m'}, \quad \forall m' \neq m.
\end{equation}
Indeed, it follows from the definition of these indices and the form of $\Lambda_{p,q}$ that, if $p_l^m = p_l^{m'}$, then we also have $q_l^m=q_l^{m'}$, which is excluded by definition.

\begin{proof}[Proof of Theorem \ref{prop rect 1}]
Let us consider the case $\gamma=\gamma_L$; the proof for $\gamma=\gamma_L\cup \gamma_R$ relies on the same kind of arguments.
We also only prove that, if system (\ref{syst rect 1D}) is approximately controllable, then so is system (\ref{syst 1}), the converse being easier.

Let $s \in \spec{\Deltavec+A^*}$ and $u \in \ker(s-(\Deltavec+A^*)) \cap \ker (B^*1_{\gamma_L}\dn)$.
With the notations previously introduced $u$ then writes
$$
u=\sum_{k=1}^{r} \sum_{m=1}^{m_{l_k}} \left(\sum_{j} \alpha_{k,j,m}w_{i_k,j}\right) \phi_{p_{l_k}^m}^{X_1}\phi_{q_{l_k}^m}^{X_2},
$$
for some $\alpha_{k,j,m} \in \C$, and we have
\begin{equation}\label{dn upu}
\sum_{k=1}^{r} \sum_{m=1}^{m_{l_k}} \beta_{k,m} \phi_{q_{l_k}^m}^{X_2}(y)=0, \quad \forall y \in (0,X_2),
\end{equation}
where we have set
$$\beta_{k,m}=-\gamma_{k,m} \left(\phi_{p_{l_k}^m}^{X_1}\right)'(0), \quad \gamma_{k,m}=B^*\left(\sum_{j} \alpha_{k,j,m} w_{i_k,j}\right).$$

Note that we can always assume that $\ker(\theta_i-A^*) \cap \ker B^* =\ens{0}$ for every $i$ since it is a necessary condition (see Lemma \ref{lemma cond nec}) for both systems (\ref{syst 1}) and (\ref{syst rect 1D}).
As a result, to prove that $u=0$ it is equivalent to show that $\gamma_{k,m}=0$ (or $\beta_{k,m}=0$) for every $k,m$.

For convenience we assume that $r=2$.
Thus (\ref{dn upu}) becomes
$$
\sum_{m=1}^{m_{l_1}} \beta_{1,m} \phi_{q_{l_1}^m}^{X_2}(y)+ \sum_{m=1}^{m_{l_2}} \beta_{2,m} \phi_{q_{l_2}^m}^{X_2}(y)=0, \quad \forall y \in (0,X_2).
$$

Using the linear independence of $\ens{\phi_q^{X_2}}_q$ in $L^2(0,X_2)$, two cases may happen.
For some given $m$, if $q_{l_1}^m \neq q_{l_2}^{m'}$ for every $m'$ then, taking also (\ref{q diff}) into account, we obtain $\beta_{1,m}=0$.
On the other hand, if there exists $m'$ such that $q_{l_1}^m = q_{l_2}^{m'}$, then this $m'$ is unique thanks to (\ref{q diff}) and we obtain that $\beta_{1,m}+\beta_{2,m'}=0$, that is
$$
-\left(\gamma_{1,m}\phi_{p_{l_1}^m}^{X_1}+\gamma_{2,m'} \phi_{p_{l_2}^{m'}}^{X_1}\right)'(0)=0.
$$
Since $q_{l_1}^m = q_{l_2}^{m'}$ we have
$$-\lambda_{p_{l_1}^m}^{X_1}+\theta_{i_1}=-\lambda_{p_{l_2}^{m'}}^{X_1}+\theta_{i_2},$$
and the assumption that system (\ref{syst rect 1D}) is approximately controllable permits to conclude that $\gamma_{1,m}=\gamma_{2,m'}=0$.

Thus, in both situations $\gamma_{1,m}=0$, and it follows that $\gamma_{2,m}=0$ (when $r>2$ we reason by induction).
\end{proof}

\begin{proof}[Proof of Theorem \ref{prop rect 2}]
Since we assume that $n=2$, the matrix $A^*$ has at most two distinct eigenvalues.
If $A^*$ has only one eigenvalue then we already know that the system is approximately controllable, see Remark \ref{rem 1vp} in section \ref{sect cond suffi syst1}.
Let us then assume that $A^*$ has two distinct eigenvalues
\begin{equation}\label{theta diff}
\theta_{i_1} \neq \theta_{i_2}.
\end{equation}

With the same notations as in the proof of Theorem \ref{prop rect 1}, let us show that is it not possible to have
$$
\left\{\begin{array}{l}
\ds \sum_{m=1}^{m_{l_1}} \beta_{1,m} \phi_{q_{l_1}^m}^{X_2}+ \sum_{m=1}^{m_{l_2}} \beta_{2,m} \phi_{q_{l_2}^m}^{X_2}=0, \\
\ds \sum_{m=1}^{m_{l_1}} \delta_{1,m} \phi_{p_{l_1}^m}^{X_1}+ \sum_{m=1}^{m_{l_2}} \delta_{2,m} \phi_{p_{l_2}^m}^{X_1}=0, \\
\end{array}\right.
$$
with $\gamma_{k,m} \neq 0$ for every $k,m$, where
$$\delta_{k,m}=\gamma_{k,m} \left(\phi_{q_{l_k}^m}^{X_2}\right)'(X_2).$$

From the first equation we see that the sets $\ens{q_{l_1}^m}_{1 \leq m \leq m_{l_1}}$ and $\ens{q_{l_2}^{m'}}_{1 \leq m' \leq m_{l_2}}$ are in bijection.
Indeed, if there exists $m$ such that $q_{l_1}^m \neq q_{l_2}^{m'}$ for every $m'$ then, using the linear independence of $\ens{\phi_q^{X_2}}_q$ in $L^2(0,X_2)$, we obtain $\beta_{1,m}=\gamma_{1,m}=0$.
Since the same fact holds for the second equation, the sets $\ens{p_{l_1}^m}_{1 \leq m \leq m_{l_1}}$ and $\ens{p_{l_2}^{m'}}_{1 \leq m' \leq m_{l_2}}$ are also in bijection.

As a consequence, denoting $M=m_{l_1} =m_{l_2}$, we have
$$\sum_{m=1}^{M}\lambda_{q_{l_1}^m}^{X_2}= \sum_{m'=1}^{M}  \lambda_{q_{l_2}^{m'}}^{X_2}, \quad \sum_{m=1}^{M}  \lambda_{p_{l_1}^m}^{X_1} =\sum_{m'=1}^{M}  \lambda_{p_{l_2}^{m'}}^{X_1},$$
so that
$$\sum_{m=1}^{M} \Lambda_{p_{l_1}^m,q_{l_1}^m}=\sum_{m'=1}^{M} \Lambda_{q_{l_2}^{m'},q_{l_2}^{m'}}.$$

Let us denote by $S$ this common value.
Since $-\Lambda_{p_{l_1}^m,q_{l_1}^m}+\theta_{i_1}=-\Lambda_{q_{l_2}^{m},q_{l_2}^{m}}+\theta_{i_2}$ for every $m$, if we sum we obtain
$$-S+M \theta_{i_1}=-S+M\theta_{i_2},$$
and thus
$$\theta_{i_1}=\theta_{i_2},$$
a contradiction with our assumption (\ref{theta diff}).
\end{proof}

\begin{proof}[Proof of Theorem \ref{prop rect 3}]
We provide an example of system with 4 equations for which the condition $\ker(\theta_i-A^*) \cap \ker B^* =\ens{0}$ holds for every $i$ and that is not approximately controllable on $\gamma=\gamma_L \cup \gamma_T$.
This example can easily be generalized to the case $n>4$.

We take $X_1=X_2=\pi$, so that the eigenvalues of $\Delta$ are simply
$$-\Lambda_{p,q}=-p^2-q^2,$$
and we choose
$$A=
\left(\begin{array}{cccc}
0 & 0 & 0 & 0 \\
1 & 0 & 0 & 120 \\
0 & 1 & 0 & -1 \\
0 & 0 & 1&  -10
\end{array} \right),
\quad
B= \left(\begin{array}{c} 1 \\ 0 \\ 0 \\ 0 \end{array}\right).
$$

We can check that
$$\spec{A^*}=\ens{\theta_1=-8,\theta_2=-5,\theta_3=3,\theta_4=0}.$$
$$\ker(\theta_i-A^*) \cap \ker B^*=\ens{0}, \quad i=1,2,3,4.$$

Now, observe that we have the relation
$$-\Lambda_{1,1}+\theta_1=-\Lambda_{2,1}+\theta_2=-\Lambda_{2,3}+\theta_3=-\Lambda_{1,3}+\theta_4=-10.$$

In view of this relation we define
$$u= \phi_{1,1}-\frac{1}{2}\phi_{2,1}+\frac{1}{6}\phi_{2,3}-\frac{1}{3}\phi_{1,3}.$$
Clearly $u \neq 0$.
Let us show that however $\dn u =0$ on $\gamma_L \cup \gamma_T$.
This will contradict Theorem \ref{prop B vect}.

Taking into account that $\left(\phi_p^{X_1}\right)'(0)= p\sqrt{\frac{2}{\pi}}$, for $x_2 \in (0,\pi)$ we have
\begin{multline*}
-\partial_{x_1} u(0,x_2) =-\underbrace{\left(\left(\phi_1^{X_1}\right)'(0)-\frac{1}{2}\left(\phi_2^{X_1}\right)'(0)\right)}_{=0} \phi_1^{X_2}(x_2) \\
-\underbrace{\left(\frac{1}{6}\left(\phi_2^{X_1}\right)'(0)-\frac{1}{3}\left(\phi_1^{X_1}\right)'(0) \right)}_{=0} \phi_3^{X_2}(x_2),
\end{multline*}
so that indeed $\dn u=0$ on $\gamma_L$.

In the same way, for $x_1 \in (0,\pi)$ we have
\begin{multline*}
\partial_{x_2} u(x_1,\pi)=\phi_1^{X_1}(x_1) \underbrace{\left(\left(\phi_1^{X_2}\right)'(\pi)-\frac{1}{3}\left(\phi_3^{X_2}\right)'(\pi)\right)}_{=0} \\
+\phi_2^{X_1}(x_1) \underbrace{\left(-\frac{1}{2}\left(\phi_1^{X_2}\right)'(\pi)+\frac{1}{6}\left(\phi_3^{X_2}\right)'(\pi)\right)}_{=0},
\end{multline*}
and thus $\dn u=0$ also on $\gamma_T$.
\end{proof}

\section{Results for the second system}\label{result syst2}

We now turn out to the results concerning the second system
$$
\left\{\begin{array}{ll}
\ds \pt y_1 =  \Delta y_1 &\mbox{ in } (0,T) \times \Omega. \\
\ds \pt y_2 =  \Delta y_2 + G(x) \cdot \nabla y_1 + a(x) y_1&\mbox{ in } (0,T) \times \Omega. \\
y_1 = 1_{\gamma}g, \quad y_2=0 &\mbox{ on } (0,T) \times \partial\Omega. \\
y_1(0)= y_{1,0}, \quad y_2(0)=y_{2,0} & \mbox{ in } \Omega.
\end{array}\right.
$$

As mentionned in the introduction, it is known that this system is approximately controllable at time $T$ if and only if its adjoint system
$$
\left\{\begin{array}{ll}
\ds \pt z_1 =  \Delta z_1 -G(x) \cdot \nabla z_2 + (a(x)-\dive{G}(x))z_2 &\mbox{ in } (0,T) \times \Omega. \\
\ds \pt z_2 =  \Delta z_2 &\mbox{ in } (0,T) \times \Omega. \\
z_1 = 0, \quad z_2=0 &\mbox{ on } (0,T) \times \partial\Omega. \\
z_1(0)= z_{1,0}, \quad z_2(0)=z_{2,0} & \mbox{ in } \Omega.
\end{array}\right.
$$
has the unique continuation property
$$
\forall  z_{1,0}, z_{2,0} \in H^1_0(\Omega), \quad
\bigg( 1_{\gamma}\dn z_1(t)=0 \mbox{ for a.e. } t \in (0,T) \bigg)
\Longrightarrow z_{1,0}=z_{2,0}=0.
$$
For commodity, let us denote
$$\opQ=-G(x) \cdot \nabla+ (a(x)-\dive{G}(x)).$$
Thus, we want to apply Theorem \ref{fat} to the operators
$$
\opA=\left(\begin{array}{cc} \Delta & \opQ \\  0 & \Delta \end{array}\right), \quad \dom{\opA}=H^2(\Omega)^2 \cap H^1_0(\Omega)^2,
$$
and
$$
\opC=\left(\begin{array}{cc} 1_{\gamma}\dn & 0 \end{array}\right), \quad \dom{\opC}=H^2(\Omega)^2 \cap H^1_0(\Omega)^2.
$$
Again, by using a perturbation argument we can check that $\opA$ generates an analytic semigroup and indeed satisfies the hypothesis of Theorem \ref{fat}.
The operator $\opC$ is the same as for the first system we studied, taking $n=2$ and $B=\left(\begin{array}{c} 1 \\ 0 \end{array}\right)$.

As a consequence, this system is approximately controllable if and only if
\begin{equation}\label{fat syst2}
\ker \left(\begin{array}{cc} s-\Delta & -\opQ \\  0 & s-\Delta \end{array}\right) \cap \ker(\begin{array}{cc} 1_{\gamma}\dn & 0 \end{array})=\ens{0}, \quad \forall s \in \C.
\end{equation}

It is not difficult to see that the spectrum of $\opA$ is
$$\spec{\opA}=\ens{-\lambda_l}_{l},$$
(we refer to the introduction for the notations) and that its eigenspaces can be decomposed as follows:
$$\ker(-\lambda_l-\opA)=U_l \oplus V_l,$$
with
$$
U_l=\ens{\left(\begin{array}{c} u \\ 0 \end{array}\right)}_{u \in \ker(-\lambda_l-\Delta)}, \quad
V_l=\ens{\left(\begin{array}{c} \opS_l \opQ v \\ v \end{array}\right)}_{v \in \ker(-\lambda_l -\Delta) \cap \ker \opPl \opQ},
$$
where $\opS_l: f \in \ker \opPl \longmapsto u \in \ker \opPl$ with $u$ the unique solution (in $\ker \opPl$) of the equation $(-\lambda_l-\Delta) u =f$.

\subsection{A sufficient condition}

The following theorem is, in some sense, the analogue of Theorem \ref{prop cond suffi} in section \ref{sect cond suffi syst1}.
This also recovers \cite[Theorem 1.5]{KdT}.

\begin{theorem}\label{prop cond suffi syst2}
Assume that
\begin{equation}\label{cond suffi syst2}
\ker(-\lambda_l -\Delta) \cap \ker \opPl \opQ=\ens{0}, \quad \forall l.
\end{equation}
Then, the ND system
$$
\left\{\begin{array}{ll}
\ds \pt y_1 =  \Delta y_1 &\mbox{ in } (0,T) \times \Omega. \\
\ds \pt y_2 =  \Delta y_2 + G(x) \cdot \nabla y_1 + a(x) y_1&\mbox{ in } (0,T) \times \Omega. \\
y_1 = 1_{\gamma}g, \quad y_2=0 &\mbox{ on } (0,T) \times \partial\Omega. \\
y_1(0)= y_{1,0}, \quad y_2(0)=y_{2,0} & \mbox{ in } \Omega.
\end{array}\right.
$$
is approximately controllable.
\end{theorem}

\begin{remark}
Condition (\ref{cond suffi syst2}) can be reformulated into the following rank condition:
\begin{equation}\label{rank cond syst2}
\rank \left(\begin{array}{ccc} \ps{\opQ\phi_{l,1}}{\phi_{l,1}}{L^2(\Omega)} & \cdots & \ps{\opQ\phi_{l,1}}{\phi_{l,m_l}}{L^2(\Omega)} \\ \vdots & & \vdots \\ \ps{\opQ\phi_{l,m_l}}{\phi_{l,1}}{L^2(\Omega)} & \cdots & \ps{\opQ\phi_{l,m_l}}{\phi_{l,m_l}}{L^2(\Omega)} \end{array} \right)=m_l, \quad \forall l.
\end{equation}
\end{remark}

\begin{proof}[Proof of Theorem \ref{prop cond suffi syst2}]
The assumption (\ref{cond suffi syst2}) means that $V_l=\ens{0}$ for every $l$, so that $\ker(-\lambda_l-\opA)=U_l$ for every $l$.
Thus, any $w \in \ker(-\lambda_l-\opA)\cap \ker \opC$ writes $w=\left(\begin{array}{c} u \\ 0 \end{array}\right)$ for some $u \in \ker(-\lambda_l-\Delta)$ and satisfies
$$Cw=1_{\gamma}\dn u =0.$$
This implies $u=0$ (see (\ref{HT diff})) and thus also $w=0$.
\end{proof}

\subsection{The 1D case}\label{sect 1D syst2}

As for Theorem \ref{1D B vect} in section \ref{sect B vect}, condition (\ref{cond suffi syst2}) turns out to be also necessary in 1D:

\begin{theorem}\label{prop 1D syst2}
The 1D system
\begin{equation}\label{syst2 1D}
\left\{\begin{array}{ll}
\ds \pt y_1 =  \Delta y_1 &\mbox{ in } (0,T) \times (0,L). \\
\ds \pt y_2 =  \Delta y_2 +G(x) \cdot \nabla y_1 + a(x) y_1&\mbox{ in } (0,T) \times (0,L). \\
y_1 = 1_{\ens{0}}g, \quad y_2=0 &\mbox{ on } (0,T) \times \ens{0,L}. \\
y_1(0)= y_{1,0}, \quad y_2(0)=y_{2,0} & \mbox{ in } (0,L).
\end{array}\right.
\end{equation}
is approximately controllable if and only if
\begin{equation}\label{caract syst2 1D}
\int_{0}^{L} \left(-\frac{1}{2}G'(x)+a(x)\right) \abs{\phi_l(x)}^2 dx \neq 0, \quad \forall l.
\end{equation}
\end{theorem}

\begin{remark}\label{rem syst2 au bord}
This theorem provides an easy condition to check whether this system is approximately controllable or not.
For instance, for $a=0$ and $G$ constant, the corresponding system is not approximately controllable.
\end{remark}

\begin{proof}[Proof of Theorem \ref{prop 1D syst2}]
Since $m_l=1$ for every $l$ ($N=1$), condition (\ref{rank cond syst2}) now reads as
$$\ps{Q\phi_l}{\phi_l}{L^2(\Omega)} \neq 0, \quad \forall l,$$
which gives the same condition than (\ref{caract syst2 1D}) after an integration by part on the gradient term.

From Theorem \ref{prop cond suffi syst2} we already know that this condition is sufficient.
Let us prove that it is also necessary in our case.

Assume that this condition does not hold or, equivalently, that (\ref{cond suffi syst2}) does not hold.
Then, for some $l$, there exists at least one eigenfunction of $\opA$ associated with $-\lambda_l$ in $U_l$, say $\left(\begin{array}{c} u \\ 0 \end{array}\right)$, and another one in $V_l$, say $\left(\begin{array}{c} \opS_l\opQ v \\ v \end{array}\right)$.
If $(\opS_l\opQ v)'(0)=0$ then condition (\ref{fat syst2}) already fails.
On the other hand, if $(\opS_l\opQ v)'(0) \neq 0$, then condition (\ref{fat syst2}) also fails because of the following relation
$$\left(\frac{1}{u'(0)} u-\frac{1}{\left(\opS_l\opQ v\right)'(0)} \opS_l\opQ v \right)'(0) =0.$$
\end{proof}

\section{Further results: distributed controllability}\label{article 2/further results}

All along this work we were interested in the boundary controllability problem but let us mention that the method also works for distributed controllability.
For instance, we can recover the result of \cite{AKBDGB} concerning system (\ref{syst 1}).
We can also obtain the following result:

\begin{theorem}\label{prop distri}
Let $\omega$ be a nonempty open subset of $\Omega$.
Assume that $\Omega$ is connected and $G$ and $a$ are real analytic functions in $\Omega$.
Then, the ND system
$$
\left\{\begin{array}{ll}
\ds \pt y_1 =  \Delta y_1 +1_{\omega}g &\mbox{ in } (0,T) \times \Omega. \\
\ds \pt y_2 =  \Delta y_2 + G(x) \cdot \nabla y_1 + a(x) y_1&\mbox{ in } (0,T) \times \Omega. \\
y_1 =0, \quad y_2=0 &\mbox{ on } (0,T) \times \partial\Omega. \\
y_1(0)= y_{1,0}, \quad y_2(0)=y_{2,0} & \mbox{ in } \Omega.
\end{array}\right.
$$
is approximately controllable if and only if
\begin{equation}\label{hyp Q}
\ker(-\lambda_l-\Delta) \cap \ker \opQ=\ens{0}, \quad \forall l,
\end{equation}
where we recall that $\opQ=-G(x) \cdot \nabla+ (a(x)-\dive{G}(x))$.
\end{theorem}

To the author knowledge, \cite{BCGdT}, \cite{Gue} and \cite{KdT} are the only works for the distributed null and approximate controllability of this system.
However in these papers, even for the case $a=0$ and $G$ constant, a geometric assumption or a particular form of $G$ is required (see \cite[Theorem 2.1]{BCGdT} and \cite[Theorem 4]{Gue}), while \cite[Theorem 1.5]{KdT} can not be applied.

\begin{proof}[Proof of Theorem \ref{prop distri}]
This time the observation operator is $\opC=\left(\begin{array}{cc} 1_{\omega} & 0 \end{array}\right)$ (it is a bounded operator on $L^2(\Omega)^2$).
Let $w \in \ker(-\lambda_l-\opA) \cap \ker \opC$.
Thus, $w$ writes $w=\left(\begin{array}{c} u \\ 0 \end{array}\right) + \left(\begin{array}{c} \opS_l \opQ v \\ v \end{array}\right)$ for some $u \in \ker(-\lambda_l-\Delta)$ and $v \in \ker(-\lambda_l -\Delta) \cap \ker \opPl \opQ$, and satisfies
\begin{equation}\label{condition}
1_{\omega} \left(u + \opS_l \opQ v\right)=0.
\end{equation}
Since $v$ is an analytic function, so is $\opQ v$.
Thus, $\opS_l \opQ v$ is an analytic function as solution of an elliptic partial differential equation with analytic data (see for instance \cite[Theorem 7.5.1]{Hor}).
Note that $u$ is also analytic.
Thus, (\ref{condition}) is equivalent to
$$u + \opS_l \opQ v=0.$$
This implies that $u=0$ since $u=-\opS_l \opQ v \in \ker\opPl$ and $u \in \ker(-\lambda_l-\Delta)=\image{\opPl}$.
Thus, $\opS_l \opQ v=-u=0$ and it follows that $\opQ v=0$ (see the definition of $\opS_l$).
This implies $v=0$ if and only if the hypothesis holds.
\end{proof}

Let us illustrate this result with $a=0$ and $G \neq 0$ constant. This means that we take $\opQ=-G \cdot \nabla$.
We can verify that this $\opQ$ satisfies (\ref{hyp Q}) since
$$\forall u \in H^1_0(\Omega), \quad G \cdot \nabla u = 0 \mbox{ in } \Omega \Longrightarrow u=0.$$
Indeed, set $v(x)=e^{G \cdot x} u(x)$.
We have $v \in H^1_0(\Omega)$ and (using the hypothesis on $u$)
$$G \cdot \nabla v = \abs{G}^2 v.$$ 
Multiplying this equality by $v$ and integrating by parts we obtain
$$\abs{G}^2 \int_{\Omega} \abs{v(x)}^2 \, dx=0.$$
This implies that $v=0$ and thus also $u=0$.

It follows from Theorem \ref{prop distri} that the ND system
$$
\left\{\begin{array}{ll}
\ds \pt y_1 =  \Delta y_1 +1_{\omega}g &\mbox{ in } (0,T) \times \Omega. \\
\ds \pt y_2 =  \Delta y_2 + G \cdot \nabla y_1 &\mbox{ in } (0,T) \times \Omega. \\
y_1 =0, \quad y_2=0 &\mbox{ on } (0,T) \times \partial\Omega. \\
y_1(0)= y_{1,0}, \quad y_2(0)=y_{2,0} & \mbox{ in } \Omega.
\end{array}\right.
$$
is approximately controllable.

\begin{remark}
This establishes another difference between distributed and boundary controllability for parabolic systems.
Indeed, let us recall that we have seen that this same system in 1D with a boundary control is not approximately controllable, see Remark \ref{rem syst2 au bord}.
\end{remark}

\appendix
\section{Proof of Theorem \ref{fat}}\label{annexe 1}

For the sake of completeness we give here the proof of Theorem \ref{fat}.
We recall that this proof is just adapted from the one in \cite{Fat} in order to deal with relatively bounded observation operators.

Let us recall the notations and assumptions.
$H$ and $U$ are complex Hilbert spaces, $\opA : \dom{\opA} \subset H \longrightarrow H$ generates a strongly continuous semigroup on $H$, has a compact resolvent, the system of root vectors of $\opA^*$ is complete in $H$, and $\opC: \dom{\opC} \subset H \longrightarrow U$ is relatively bounded with respect to $\opA$.

We denote by $\rho(\opA)$ the resolvent set of our closed linear operator $\opA$ and, for $\lambda \in \rho(\opA)$, $\res{\lambda}{\opA}=(\lambda-\opA)^{-1}$ the resolvent operator.

Since $\opA$ has a compact resolvent, its spectrum $\spec{\opA}=\C \backslash \rho(\opA)$ consists in a sequence of isolated points, say $\ens{\mu_j}_j$.
In particular $\rho(\opA)$ is path connected.
We have $\spec{\opA^*}=\conj{\spec{\opA}}=\ens{\conj{\mu_j}}_j$.

Let now $C_j \subset \rho(\opA)$ be a positive-oriented small circle enclosing $\mu_j$ and such that no other eigenvalue than $\mu_j$ lies inside this circle.
For every $j$ we define the spectral projection
$$
\begin{array}{rccl}
\proj: & H & \longrightarrow & H \\
& u & \longmapsto &\ds \frac{1}{2\pi i} \int_{C_j} \res{\xi}{\opA}u \, d\xi.
\end{array}
$$
The operator $\proj$ is a bounded linear operator and one can prove that the range of this operator is exactly the root subspace of $\opA$ associated with $\mu_j$, i.e. $\ker(\mu_j-\opA)^{\tau_j}$, where $\tau_j$ is the smallest indice $k$ such that $\ker(\mu_j-\opA)^{k+1}=\ker(\mu_j-\opA)^k$.
For a proof of this fact we refer to \cite[2 Lemma, Chapter XIX]{DS}.
A computation shows that
$$\proj^*= \frac{1}{2\pi i} \int_{\conj{C_j}} \res{\xi}{\opA^*} \, d\xi,$$
where $\conj{C_j}$ is the circle centered in $\conj{\mu_j}$ with the same radius as $C_j$.
Since there are no eigenvalue of $\opA^*$ except $\conj{\mu_j}$ inside the circle $\conj{C_j}$, the range of this operator is exactly the root subspace of $\opA^*$ associated with $\conj{\mu_j}$.

Let us now recall some properties of semigroups.
Since $\opA$ generates a strongly continuous semigroup $\opS(t)$ on $H$, we know that there exists $M>0$ and $\omega_0 \in \R$ such that
$$\norm{\opS(t)}{\mathcal{L}(H)} \leq M e^{\omega_0 t}, \quad \forall t \geq 0.$$
Moreover, for every $z_0 \in \dom{\opA}$, we have $\opS(t)z_0 \in \dom{\opA}$ with $\opA \opS(t)z_0=\opS(t)\opA z_0$ and the map $t \in [0,+\infty) \longmapsto \opS(t)z_0 \in \dom{\opA}$ is continuous.
Finally, the resolvent set $\rho(A)$ contains the halfplane $\ens{\lambda \in \C \, | \, \re \lambda >\omega_0}$.
For a proof of these facts we refer to \cite[Proposition 5.5, Chapter I]{EN}, \cite[Lemma 1.3, Chapter II]{EN} and \cite[Theorem 1.10, Chapter II]{EN}, respectively.

\begin{lemma}[Corollary 2.2 of \cite{Fat}]\label{lem1}
Let $z_0 \in \dom{\opA}$ be fixed.
The three following properties are equivalent:
\begin{enumerate}
\item\label{semi} $\opC\opS(t)z_0=0$ for a.e. $t \in (0,+\infty)$.
\item\label{reso1} $\opC\res{\lambda}{\opA} z_0=0$ for every $\lambda \in \C$ with $\re{\lambda}> \omega_0$.
\item\label{reso2} $\opC\res{\lambda}{\opA} z_0=0$ for every $\lambda \in \rho(\opA)$.
\end{enumerate}
\end{lemma}

\begin{proof}[Proof]
Recall that the resolvent of an operator can be represented as the Laplace transform of the semigroup it generates for $\re{\lambda}>\omega_0$ (see for instance \cite[Theorem 1.10, Chapter II]{EN}):
$$\begin{array}{rcl}
\ds \res{\lambda}{\opA}z_0 &=&\ds  \int_{0}^{+\infty} e^{-\lambda t} \opS(t)z_0 \, dt \\
&=& \ds \lim_{j \to +\infty} \int_{0}^{j} e^{-\lambda t} \opS(t)z_0 \, dt \quad (\mbox{limit in } H).
\end{array}
$$
Actually, this limit can also be considered in the sense of $\dom{\opA}$.
Indeed,
$$\int_{0}^{j} e^{-\lambda t} \opS(t)z_0 \, dt \in \dom{\opA},$$
with
\begin{multline*}
\opA\int_{0}^{j} e^{-\lambda t} \opS(t)z_0 \, dt
=\int_{0}^{j} e^{-\lambda t} \opA\opS(t)z_0 \, dt
= \int_{0}^{j} e^{-\lambda t} \opS(t)\opA z_0 \, dt, \\
\xrightarrow[j \to +\infty]{} \res{\lambda}{\opA}\opA z_0.
\end{multline*}
Since $\opC$ is bounded on $\dom{\opA}$ we obtain
$$\opC\res{\lambda}{\opA}z_0=\int_{0}^{+\infty} e^{-\lambda t} \opC\opS(t)z_0 \, dt, \quad \re{\lambda}>\omega_0.$$
It is now clear that \ref{semi}. implies \ref{reso1}. while the converse follows from the injectivity of the Laplace transform.

The remaining equivalence is a consequence of the analytic continuation of the resolvent.
\end{proof}

\begin{lemma}[Proposition 3.1 of \cite{Fat}]\label{lem2}
Let $z_0 \in \dom{\opA}$ be fixed.
If the third point of Lemma \ref{lem1} holds, then $\opC\res{\lambda}{\opA}\proj z_0=0$ for every $\lambda \in \rho(\opA)$ and every $j$.
\end{lemma}

\begin{proof}[Proof]
Let $\lambda \in \rho(\opA)$ lies outside the circle $C_j$.

The first resolvent equation $(\lambda-\xi)\res{\lambda}{\opA}\res{\xi}{\opA}=\res{\xi}{\opA}-\res{\lambda}{\opA}$ gives
$$
\begin{array}{rcl}
\res{\lambda}{\opA} \proj  z_0 &=&\ds \res{\lambda}{\opA} \frac{1}{2\pi i}\int_{C_j} \res{\xi}{\opA} z_0 \, d\xi \\
&=&\ds \frac{1}{2\pi i}\int_{C_j} \res{\lambda}{\opA}\res{\xi}{\opA} z_0 \, d\xi \\
&=&\ds -\frac{1}{2\pi i}\int_{C_j} \frac{\res{\xi}{\opA}-\res{\lambda}{\opA}}{\xi-\lambda} z_0 \, d\xi \\
&=&\ds -\frac{1}{2\pi i}\int_{C_j} \frac{\res{\xi}{\opA}}{\xi-\lambda} z_0 \, d\xi + \left(\frac{1}{2\pi i}\int_{C_j} \frac{1}{\xi-\lambda} \, d\xi\right) \res{\lambda}{\opA} z_0. \\
\end{array}
$$
Since $\lambda$ lies outside $C_j$, the second integrand is analytic in some disk enclosing $C_j$ and thus, by Cauchy's theorem, the second integral is zero.
This gives
$$
\res{\lambda}{\opA}  \proj z_0=
-\frac{1}{2\pi i}\int_{C_j} \frac{\res{\xi}{\opA}}{\xi-\lambda} z_0 \, d\xi.
$$
Once again this integral can be taken in $\dom{\opA}$.
Thus, applying $\opC$ we have
$$
\opC\res{\lambda}{\opA}  \proj z_0=
-\frac{1}{2\pi i}\int_{C_j} \frac{\opC\res{\xi}{\opA}}{\xi-\lambda}  z_0 \, d\xi.
$$
Using the assumption we obtain $\opC\res{\lambda}{\opA}  \proj  z_0=0$ for every such $\lambda$, and thus, by analytic continuation, for every $\lambda \in \rho(\opA)$.
\end{proof}

We are now ready to prove Theorem \ref{fat}.
Let us just introduce a last definition for commodity.
For a subspace $E \subset H$ invariant under $\opS(t)$, we say that the pair $(\opA,\opC)$ is observable in $E$ if
$$
\forall z_0 \in E, \quad
\bigg( \opC \opS(t)z_0=0 \mbox{ for a.e. } t \in (0,+\infty) \bigg)
\Longrightarrow z_0=0.
$$

\begin{proof}[Proof of Theorem \ref{fat}]
We will prove that the following properties are equivalent:
\begin{enumerate}
\item The pair $(\opA,\opC)$ is observable in every eigenspace of $\opA$.
\item The pair $(\opA,\opC)$ is observable in every root subspace of $\opA$.
\item The pair $(\opA,\opC)$ is observable in $\dom{\opA}$.
\end{enumerate}
It is adapted from Corollary 3.2 and Corollary 3.3 of \cite{Fat}.
We recall that the first condition is equivalent to: $\ker(s-\opA) \cap \ker \opC=\ens{0}$ for every $s \in \C$ (see Remark \ref{rem prem}).

The scheme of the proof is 1. $\Longrightarrow$ 2. $\Longrightarrow$ 3. $\Longrightarrow$ 1. (the last implication is obvious).

Assume that the pair $(\opA,\opC)$ is observable in every eigenspace.
If $z_0$ belongs to the root subspace of $\opA$ associated with $\mu_j$, then $\opS(t)z_0$ is a polynomial in $t$, up to a factor $e^{\mu_j t}$:
$$\opS(t)z_0= e^{\mu_j t} p_j(t),$$
with
$$p_j(t)=\sum_{\sigma=0}^{\tau_j-1} a_{j,\sigma}t^{\sigma}, \quad a_{j,\sigma}=\frac{(-1)^{\sigma}}{\sigma!} (\mu_j-\opA)^{\sigma} z_0.$$
This can be seen using the uniqueness of the solution to the evolution equation satisfied by $\opS(\cdot)z_0$.
Thus, the identity $\opC\opS(\cdot)z_0=0$ reads
$$\opC(\mu_j-\opA)^{\sigma} z_0=0, \quad 0 \leq \sigma \leq \tau_j-1.$$
In particular for $\sigma=\tau_j-1$ we have
$$\opC (\mu_j-\opA)^{\tau_j-1} z_0=0.$$
Now, recall that $z_0$ lies in the root subspace $\ker(\mu_j-\opA)^{\tau_j}$, so that
$$(\mu_j-\opA)^{\tau_j-1}z_0 \in \ker(\mu_j-\opA).$$
Thus, the assumption implies that
\begin{equation}\label{aucune idée}
(\mu_j-\opA)^{\tau_j-1}z_0=0.
\end{equation}
Taking this time $\sigma=\tau_j-2$ we have
$$\opC (\mu_j-\opA)^{\tau_j-2} z_0=0,$$
and from (\ref{aucune idée}) we know that
$$(\mu_j-\opA)^{\tau_j-2}z_0\in \ker(\mu_j-\opA),$$
so that the assumption gives
$$(\mu_j-\opA)^{\tau_j-2}z_0=0.$$
Iterating this process we obtain $z_0=0$.

Assume now that the pair $(\opA,\opC)$ is observable in every root subspace and let $z_0\in \dom{\opA}$ be such that $\opC\opS(t)z_0=0$.
Applying Lemma \ref{lem1} and Lemma \ref{lem2} we obtain that $\opC\opS(t) \proj z_0=0$ for a.e. $t \in (0,+\infty)$ and every $j$.
By assumption we deduce that $\proj z_0=0$ for every $j$, that is, $z_0 \in {\left(\image{\proj^*}\right)}^{\perp}$ for every $j$.
Since the system of root vectors of $\opA^*$ is assumed to be complete in $H$, we conclude that $z_0=0$.
\end{proof}

\section*{Acknowledgments}

The author would like to thank Assia Benabdallah and Franck Boyer for their
advices and constant support. The author also thanks Mehdi Badra and Tak\'eo
Takahashi for mentioning the paper \cite{Fat}.

\bibliographystyle{plain}
\bibliography{boundary-approximate-controllability}

\end{document}